\documentclass[11pt]{amsart}
\usepackage{amssymb,mathtools,bm,graphicx}
\usepackage{booktabs}
\usepackage{enumitem}
\usepackage{pgfplots}
\usepackage{pgfplotstable}
\pgfplotsset{compat=1.18}
\pgfplotsset{
  every axis/.append style={
    legend columns=-1,
    legend cell align={center},
    legend style={at={(0.5,1.16)}, anchor=south, font=\scriptsize, draw=none, fill=none, /tikz/every even column/.append style={column sep=0.45em}},
    label style={font=\small},
    tick label style={font=\scriptsize}
  }
}
\usepgfplotslibrary{groupplots}
\newcommand{\CC}{\bm C}
\newcommand{\BB}{\bm B}
\newcommand{\DD}{\bm D}
\newcommand{\EE}{\mathcal E}
\newcommand{\HH}{\bm H}
\newcommand{\Id}{\bm I}
\newcommand{\RR}{\mathcal R}
\newcommand{\SSS}{\bm S}
\newcommand{\TT}{\bm T}

\newcommand{\nn}{\bm n}
\newcommand{\uu}{\bm u}
\newcommand{\vv}{\bm v}

\newcommand{\Wi}{\mathrm{Wi}}
\newcommand{\Rey}{\mathrm{Re}}
\newcommand{\Pe}{\mathrm{Pe}}
\newcommand{\dt}{\Delta t}
\newcommand{\dd}{\,\mathrm d}
\newcommand{\calD}{\mathcal D}
\newcommand{\calK}{\mathcal K}
\newcommand{\calO}{\mathcal O}
\newcommand{\tr}{\operatorname{tr}}
\newcommand{\divv}{\nabla\!\cdot}
\newcommand{\grad}{\nabla}

\newcommand{\Exp}{\operatorname{Exp}}

\newcommand{\norm}[1]{\left\|#1\right\|}

\makeatletter
\newcounter{manualresult}[section]
\renewcommand{\themanualresult}{\thesection.\arabic{manualresult}}
\newcommand{\manualresulthead}[3]{%
  \par\medskip\stepcounter{manualresult}%
  \protected@edef\@currentlabel{\themanualresult}%
  \noindent{\normalfont\scshape #1~\themanualresult%
  \if\relax\detokenize{#3}\relax\else\ (#3)\fi.} #2\ }
\newenvironment{theorem}[1][]{\manualresulthead{Theorem}{\itshape}{#1}}{\par\medskip}
\newenvironment{lemma}[1][]{\manualresulthead{Lemma}{\itshape}{#1}}{\par\medskip}
\newenvironment{proposition}[1][]{\manualresulthead{Proposition}{\itshape}{#1}}{\par\medskip}

\renewenvironment{proof}{\par\medskip\noindent{\itshape Proof.}\ }{\hfill$\square$\par\medskip}
\makeatother

\title[Barrier Schemes for FENE Flows]{Entropy-Compatible Barrier Schemes for Diffusive FENE Flows}
\author{Sai Peng}
\address{School of Mathematics and Computational Science, Xiangtan University, Xiangtan, Hunan, China}
\email{pscfd@xtu.edu.cn}

\begin{document}

\begin{abstract}
FENE-type conformation-tensor models impose a finite-extensibility constraint that is absent from Oldroyd--B flow: the conformation tensor must satisfy $\CC\succ0$ and $\tr\CC<L^2$.  Positive definiteness alone is therefore insufficient, since a numerical state can remain positive while crossing the singular trace barrier.  Even a trace-preserving logarithmic parametrization is not enough by itself: high-order reconstruction can remain inside the finite-extensibility domain while injecting artificial FENE entropy.  We develop and analyze a barrier-preserving entropy-compatible discretization for FENE-P type flows with polymer center-of-mass molecular diffusion and for trace-singular FENE-family closures with the same entropy structure.  The method combines a trace-barrier free energy, a finite-extensibility logarithmic parametrization, a least-damping entropy-compatible barrier-log reconstruction, molecular diffusion paired with the barrier entropy variable, compatible quadrature for polymeric work, and a scaled FENE stress variable for the small-Weissenberg limit.  For admissible discrete states we prove finite-extensibility preservation at entropy quadrature points, existence and bisection computability of the maximal entropy-admissible reconstruction parameter, a fully discrete free-energy inequality with relaxation and molecular-diffusion barrier dissipation, a quantitative AP stress closure, and a fixed-discretization Newtonian limit.  A conditional relative-entropy estimate is derived on compact subsets of the finite-extensibility domain.  Numerical diagnostics verify barrier preservation, entropy-compatible reconstruction, energy decay, AP closure, coupled velocity--pressure--stress accuracy, and high-Weissenberg robustness near the trace constraint.
\end{abstract}
\maketitle

\par\medskip\noindent\textbf{Keywords.} 
FENE-P model, finite extensibility, polymer molecular diffusion, conformation tensor, trace barrier, entropy-compatible reconstruction, asymptotic-preserving schemes
\par\medskip

\par\medskip\noindent\textbf{MSC 2020.} 
65M12, 65M60, 76A10, 76M20
\par\medskip

\section{Introduction}

The Oldroyd--B conformation tensor is constrained only by positive definiteness.  In finitely extensible nonlinear elastic models the admissible set is smaller.  If $\CC$ denotes the conformation tensor and $L^2$ is the squared maximum extension parameter, a FENE-type closure requires
\[
  \CC\in \calD_L:=\{\CC\in\mathbb S_{++}^d:\ \tr\CC<L^2\}.
\]
Equivalently, in index notation the model requires $C_{kk}<L^2$.  The coefficient multiplying the elastic stress becomes singular as $\tr\CC\uparrow L^2$.  This changes the numerical problem qualitatively.  A log-conformation update can preserve $\CC\succ0$ while still allowing $\tr\CC$ to approach or exceed $L^2$ after reconstruction, interpolation, or nonlinear iteration.  The finite-extensibility barrier must therefore be treated as part of the structure, not as an afterthought.

This paper develops a structure-preserving numerical analysis for this constraint.  We focus on the FENE-P type conformation closure
\[
  f_L(\CC)=\frac{L^2-d}{L^2-\tr\CC},\qquad
  \TT(\CC)=f_L(\CC)\CC-\Id ,
\]
where $L^2>d$ and $\TT(\Id)=0$.  The singular denominator is precisely the difficulty.  The associated entropy variable is
\[
  \HH_L(\CC)=f_L(\CC)\Id-\CC^{-1}.
\]
It is the gradient of a convex barrier energy on $\calD_L$, and it yields the same polymeric-work cancellation as the Oldroyd--B entropy variable, provided the stress tensor used in the momentum equation is the same accepted tensor used in the entropy calculation.

The central idea is to treat $\calD_L$ as the computational state space.  We combine five devices.  First, a trace-barrier entropy prevents accepted states from reaching the finite-extensibility boundary.  Second, a barrier logarithmic parametrization maps unconstrained symmetric matrices bijectively into $\calD_L$, so positivity and $\tr\CC<L^2$ can be enforced during nonlinear solves.  Third, high-order reconstructed barrier-log states are accepted only if they satisfy a FENE entropy budget.  Fourth, polymer molecular diffusion is discretized in a way that dissipates the same barrier entropy.  Fifth, the momentum equation is assembled in the scaled FENE stress
\[
  \SSS=\frac{\TT(\CC)}{\Wi},
\]
which remains regular in the Newtonian relaxation limit.  The resulting formulation is simultaneously barrier preserving, energy stable, and asymptotic preserving.

The contribution is not the introduction of the FENE-P model itself.  Rather, it is a compatibility analysis for finite-extensibility constraints at the fully discrete level.  The main results are:
\begin{enumerate}[label=(\roman*)]
\item a convex trace-barrier entropy whose gradient produces the FENE stress cancellation;
\item a barrier-log parametrization that enforces $\CC\succ0$ and $\tr\CC<L^2$ by construction;
\item an entropy-compatible barrier-log reconstruction selected by the largest admissible parameter on a logarithmic path;
\item a fully discrete free-energy inequality with relaxation and polymer molecular-diffusion dissipation;
\item an AP stress-closure estimate showing $\SSS_h=\DD(\uu_h)+\calO(\Wi)$ at fixed $h$ and $\dt$;
\item a conditional relative-entropy estimate on compact subsets of the finite-extensibility domain;
\item numerical diagnostics that test the barrier, entropy correction, AP limit, coupled pressure--stress feedback, and high-Weissenberg behavior.
\end{enumerate}

Related work on FENE closures, log-conformation variables, energy-stable viscoelastic schemes, and asymptotic-preserving relaxation discretizations is extensive; see, for example, \cite{Bird1987,FattalKupferman2004,FattalKupferman2005,HulsenFattalKupferman2005,Keunings1986,BarrettSuli2011,Boyaval2009,Jin1999,LiuYu2011}.  The present paper is closest in spirit to structure-preserving conformation-tensor methods, but the finite-extensibility constraint introduces a trace barrier that is not present in Oldroyd--B.  This extra constraint is the focus of the analysis.

\begin{table}[htbp]
\centering
\caption{Structural differences between Oldroyd--B and the FENE-P setting considered here.}
\label{tab:oldroyd-fene-position}
\small
\begin{tabular}{p{0.24\linewidth}p{0.31\linewidth}p{0.34\linewidth}}
\toprule
Issue & Oldroyd--B & FENE-P type model \\
\midrule
Admissible set & $\CC\succ0$ & $\CC\succ0$ and $\tr\CC<L^2$ \\
Entropy singularity & $\lambda_{\min}(\CC)\downarrow0$ & $\lambda_{\min}(\CC)\downarrow0$ or $\tr\CC\uparrow L^2$ \\
Stress law & $\CC-\Id$ & $f_L(\CC)\CC-\Id$ with singular $f_L$ \\
Solver parametrization & log-conformation suffices for positivity & log-conformation must be combined with a trace barrier \\
Reconstruction & log reconstruction remains in $\mathbb S_{++}^d$ & barrier-log reconstruction remains in $\calD_L$ and is filtered by FENE entropy \\
AP variable & $(\CC-\Id)/\Wi$ & $(f_L(\CC)\CC-\Id)/\Wi$ \\
\bottomrule
\end{tabular}
\end{table}

The paper retains the entropy identity, finite-extensibility preservation, entropy-compatible barrier-log reconstruction, the discrete energy estimate, the AP closure, and the core benchmarks.  This organization keeps the finite-extensibility mechanism visible while connecting each diagnostic to a structural claim.

\subsection*{Position relative to logarithmic and factor reconstructions}

Classical log-conformation variables \cite{FattalKupferman2004,FattalKupferman2005,HulsenFattalKupferman2005} and square-root or symmetric-factor variables \cite{BalciThomasesRenardy2011} are designed primarily to preserve positive definiteness and to improve robustness in strongly stretched flows.  For FENE-type models this is only the first admissibility condition.  A positive tensor with $\tr\CC\ge L^2$ is outside the model domain, and a positive tensor with $\tr\CC<L^2$ may still be unacceptable if a reconstruction step injects a mesh-scale burst of FENE entropy.  Table~\ref{tab:reconstruction-comparison} summarizes the distinction.

Richter, Iaccarino, and Shaqfeh \cite{RichterIaccarinoShaqfeh2010} used a different and important finite-extensibility safeguard in DNS of FENE-P flow past a cylinder: the trace equation is advanced first and the positive root of a scalar algebraic relation is selected so that the FENE denominator remains positive.  That idea is a trace-first admissibility update.  The present construction has a different purpose.  It treats the whole tensor as the constrained state variable, enforces $\CC\succ0$ and $\tr\CC<L^2$ simultaneously by a barrier-log map, and then accepts high-order reconstructed states only through a FENE entropy budget.  Thus the guarantee is not only denominator positivity in a component update, but compatibility with the fully discrete free-energy inequality.

\begin{table}[htbp]
\centering
\caption{Comparison of reconstruction variables for FENE-type conformation tensors.}
\label{tab:reconstruction-comparison}
\small
\setlength{\tabcolsep}{3pt}
\begin{tabular}{p{0.19\linewidth}p{0.18\linewidth}p{0.18\linewidth}p{0.33\linewidth}}
\toprule
Reconstruction & Positivity & Trace barrier & Entropy compatibility \\
\midrule
Standard log & built in & not built in & not automatic for high-order states \\
Square root or factor & built in & not built in & not automatic for high-order states \\
Trace clipping & enforceable & enforceable & generally not variationally consistent \\
Barrier log & built in & built in & requires entropy acceptance step \\
Barrier log with \eqref{eq:theta-star} & built in & built in & enforced by a least-damping FENE budget \\
\bottomrule
\end{tabular}
\end{table}

Thus the proposed reconstruction has two layers.  The barrier-log map enforces the geometry of the FENE state space, while the entropy acceptance rule enforces compatibility with the free-energy estimate.  Both layers are needed for the fully discrete theorem below.

\section{FENE-P model with polymer molecular diffusion}

Let $b=L^2>d$ and define
\[
  f_b(\CC)=\frac{b-d}{b-\tr\CC},\qquad
  \calD_b=\{\CC\in\mathbb S_{++}^d:\tr\CC<b\}.
\]
The molecularly diffusive FENE-P type system is written as
\begin{align}
  \Rey(\partial_t\uu+\uu\cdot\grad\uu)+\grad p-\beta\Delta\uu
  &=(1-\beta)\divv\SSS+\bm f,                                      \label{eq:mom}\\
  \divv\uu&=0,                                                       \label{eq:div}\\
  \partial_t\CC+\uu\cdot\grad\CC-(\grad\uu)\CC-\CC(\grad\uu)^T
  &=-\frac1{\Wi}\TT(\CC)+\varepsilon\Delta\CC,                       \label{eq:conf}\\
  \TT(\CC)&=f_b(\CC)\CC-\Id,\qquad
  \SSS=\frac{\TT(\CC)}{\Wi}.                                         \label{eq:stress}
\end{align}
The normalization $f_b(\Id)=1$ makes $\CC=\Id$ the equilibrium conformation.  The parameter $\varepsilon=\Pe_p^{-1}$ represents polymer center-of-mass molecular diffusion in the constant-polymer-density FENE-P closure.  Periodic boundary conditions or the no-flux condition
\[
  \nabla\CC\,\nn=0
\]
are assumed for this diffusive flux.  If the polymer number density is not constant, the closure must be augmented by a concentration equation; that extension is outside the present analysis.  The point here is that molecular diffusion must be paired with the FENE barrier entropy rather than treated as an arbitrary componentwise smoothing operator.

The elastic entropy is
\[
  \Phi_b(\CC)
  =
  -\log\det\CC-(b-d)\log\left(\frac{b-\tr\CC}{b-d}\right).
  \label{eq:fene-entropy}
\]
It satisfies $\Phi_b(\Id)=0$ and
\[
  D\Phi_b(\CC)[\BB]
  =
  \left(f_b(\CC)\Id-\CC^{-1}\right):\BB
  =\HH_b(\CC):\BB .
\]
The Hessian is positive:
\[
  D^2\Phi_b(\CC)[\BB,\BB]
  =
  \tr(\CC^{-1}\BB\CC^{-1}\BB)
  +
  \frac{b-d}{(b-\tr\CC)^2}\bigl(\tr\BB\bigr)^2 .
  \label{eq:hessian}
\]
Thus $\Phi_b$ is convex on $\calD_b$ and blows up at both parts of the boundary: $\lambda_{\min}(\CC)\downarrow0$ and $\tr\CC\uparrow b$.

\begin{lemma}[trace-barrier buffer]
Let $\CC\in\calD_b$ and $\Phi_b(\CC)\le M$.  Then there is a positive function $\delta_b(M)$ such that
\[
  b-\tr\CC\ge \delta_b(M)>0 .
\]
One possible choice is
\[
  \delta_b(M)
  =
  (b-d)\exp\!\left[-\frac{M+d\log(b/d)}{b-d}\right].
\]
\end{lemma}

\begin{proof}
Since $\det\CC\le(\tr\CC/d)^d\le(b/d)^d$, one has $-\log\det\CC\ge-d\log(b/d)$.  From the entropy bound,
\[
  -(b-d)\log\left(\frac{b-\tr\CC}{b-d}\right)
  \le M+d\log(b/d),
\]
which gives the stated lower bound on $b-\tr\CC$.
\end{proof}

\begin{lemma}[relaxation dissipation]
For every $\CC\in\calD_b$,
\[
  \TT(\CC):\HH_b(\CC)
  =
  \sum_{i=1}^d \lambda_i\left(f_b(\CC)-\lambda_i^{-1}\right)^2
  \ge0,
\]
where $\lambda_i$ are the eigenvalues of $\CC$.
\end{lemma}

\begin{proof}
Because $f_b(\CC)$ is a scalar function of $\tr\CC$, it commutes with $\CC$.  Diagonalizing $\CC$ gives
\[
  \bigl(f_b\lambda_i-1\bigr)\bigl(f_b-\lambda_i^{-1}\bigr)
  =
  \lambda_i\bigl(f_b-\lambda_i^{-1}\bigr)^2 .
\]
Summing over $i$ proves the claim.
\end{proof}

\begin{lemma}[polymer molecular-diffusion entropy dissipation]
Assume periodic boundary conditions or $\nabla\CC\,\nn=0$ on $\partial\Omega$.  For every smooth admissible conformation field,
\[
  \int_\Omega \Delta\CC:\HH_b(\CC)\,\dd x
  =
  -\int_\Omega \mathcal I_b(\CC)\,\dd x ,
\]
where
\[
  \mathcal I_b(\CC)
  =
  \sum_{j=1}^d
  \left[
  \tr\!\left(\CC^{-1}\partial_j\CC\,\CC^{-1}\partial_j\CC\right)
  +
  \frac{b-d}{(b-\tr\CC)^2}\bigl(\partial_j\tr\CC\bigr)^2
  \right]\ge0 .
\]
\end{lemma}

\begin{proof}
Integrating by parts gives
\[
  \int_\Omega \Delta\CC:\HH_b(\CC)\,\dd x
  =
  -\sum_{j=1}^d\int_\Omega
  \partial_j\CC:D\HH_b(\CC)[\partial_j\CC]\,\dd x .
\]
Since $D\HH_b=D^2\Phi_b$, the Hessian formula \eqref{eq:hessian} gives the stated expression.
\end{proof}

\subsection*{Trace-singular FENE-family closures}

Although the formulas below are written for the FENE-P coefficient $f_b$, the same algebra applies to a larger FENE-family class.  Let
\[
  \TT_g(\CC)=g(\tr\CC)\CC-\Id,\qquad
  \HH_g(\CC)=g(\tr\CC)\Id-\CC^{-1},
\]
where $g\in C^1([d,b))$, $g(d)=1$, $g'(s)\ge0$, and $g(s)\to\infty$ as $s\uparrow b$.  Define
\[
  G_g(s)=\int_d^s g(r)\,\dd r,\qquad
  \Phi_g(\CC)=-\log\det\CC+G_g(\tr\CC).
\]
The FENE-P entropy corresponds to $g=f_b$ and
$G_g(s)=-(b-d)\log((b-s)/(b-d))$.

\begin{proposition}[common entropy mechanism for FENE-family closures]
For every $\CC\in\calD_b$ and every symmetric $\BB$,
\[
\begin{aligned}
  D\Phi_g(\CC)[\BB]&=\HH_g(\CC):\BB,\\
  D^2\Phi_g(\CC)[\BB,\BB]
  &=
  \tr(\CC^{-1}\BB\CC^{-1}\BB)+g'(\tr\CC)(\tr\BB)^2 .
\end{aligned}
\]
Moreover,
\[
  \TT_g(\CC):\HH_g(\CC)
  =
  \sum_{i=1}^d \lambda_i\left(g(\tr\CC)-\lambda_i^{-1}\right)^2\ge0,
\]
and, for $\divv\uu=0$,
\[
  -\big((\grad\uu)\CC+\CC(\grad\uu)^T\big):\HH_g(\CC)
  =
  -2\TT_g(\CC):\grad\uu .
\]
\end{proposition}

\begin{proof}
The derivative and Hessian follow from the chain rule and the derivative of $-\log\det\CC$.  Since $g(\tr\CC)$ is scalar, $\TT_g$ and $\HH_g$ are diagonal in the eigenbasis of $\CC$, giving the relaxation identity.  The stretching identity follows from
$\CC:\grad\uu=\CC:\DD(\uu)$ and $\Id:\grad\uu=\divv\uu=0$.
\end{proof}

This proposition is the structural reason for calling the method FENE-type rather than only FENE-P.  Once a closure has the trace-singular entropy variable $\HH_g$, the barrier-log state space, entropy-compatible reconstruction, and discrete coupling argument below carry over with $f_b$ replaced by $g(\tr\CC)$.  The numerical section uses FENE-P because it is the standard Peterlin closure and makes the finite-extensibility denominator explicit.

\section{Continuous energy law}

The FENE trace barrier is compatible with the usual polymeric-work cancellation.  Testing \eqref{eq:mom} by $\uu$ gives the polymeric work term
\[
  - (1-\beta)\int_\Omega \SSS:\grad\uu\,\dd x
  =
  -\frac{1-\beta}{\Wi}\int_\Omega \TT(\CC):\grad\uu\,\dd x .
\]
Testing \eqref{eq:conf} by $\HH_b(\CC)$ yields the stretching contribution
\[
  -\big((\grad\uu)\CC+\CC(\grad\uu)^T\big):\HH_b(\CC).
\]
Since $\HH_b(\CC)=f_b\Id-\CC^{-1}$ and $\divv\uu=0$,
\[
\begin{aligned}
  &-\big((\grad\uu)\CC+\CC(\grad\uu)^T\big):(f_b\Id-\CC^{-1})\\
  &\qquad
  =
  -2 f_b\CC:\grad\uu+2\Id:\grad\uu
  =
  -2\TT(\CC):\grad\uu .
\end{aligned}
\]
Multiplication by $(1-\beta)/(2\Wi)$ therefore cancels the polymeric work exactly.

\begin{theorem}[continuous barrier energy identity]
Assume a smooth solution with $\CC(t,x)\in\calD_b$ and periodic or entropy-compatible boundary conditions.  Then
\[
\begin{aligned}
  \frac{\dd}{\dd t}\left[
  \frac{\Rey}{2}\norm{\uu}^2_{L^2}
  +\frac{1-\beta}{2\Wi}\int_\Omega\Phi_b(\CC)\,\dd x
  \right]
  &+\beta\norm{\grad\uu}_{L^2}^2\\
  &+\frac{1-\beta}{2\Wi^2}
  \int_\Omega\TT(\CC):\HH_b(\CC)\,\dd x\\
  &+\frac{(1-\beta)\varepsilon}{2\Wi}
  \int_\Omega \mathcal I_b(\CC)\,\dd x
  =
  (\bm f,\uu),
\end{aligned}
\label{eq:continuous-energy}
\]
where $\mathcal I_b$ is the molecular-diffusion entropy density in Lemma~2.3.
\end{theorem}

\begin{proof}
The kinetic-energy identity follows from the incompressibility cancellation of the convection term.  The entropy derivative follows from $D\Phi_b=\HH_b$.  The stretching contribution cancels the polymeric work as shown above.  Relaxation gives the nonnegative term in Lemma~2.2, and polymer molecular diffusion gives the nonnegative barrier dissipation in Lemma~2.3.  Combining the terms gives \eqref{eq:continuous-energy}.
\end{proof}

The trace-barrier term in $\mathcal I_b$ is the new contribution relative to Oldroyd--B.  It penalizes gradients of $\tr\CC$ more strongly near the finite-extensibility boundary.  This term is also the reason that a discrete diffusion operator should be paired with the same barrier entropy rather than treated as a componentwise Laplacian detached from the entropy calculation.

\section{Barrier logarithmic parametrization}

A standard log-conformation parametrization $\CC=\Exp\Psi$ enforces $\CC\succ0$ but does not enforce $\tr\CC<b$.  We use instead the map
\[
  \mathcal B_b(\Psi)
  =
  \frac{b\,\Exp\Psi}{1+\tr(\Exp\Psi)},
  \qquad \Psi=\Psi^T .
  \label{eq:barrier-log-map}
\]
Then $\mathcal B_b(\Psi)\in\calD_b$ for every symmetric $\Psi$, because
\[
  \tr\mathcal B_b(\Psi)
  =
  \frac{b\,\tr(\Exp\Psi)}{1+\tr(\Exp\Psi)}<b .
\]
The map is onto $\calD_b$.  If $\CC\in\calD_b$, set
\[
  \Exp\Psi=\frac{\CC}{b-\tr\CC}.
\]
Then
\[
  \mathcal B_b(\Psi)
  =
  \frac{b\,\CC/(b-\tr\CC)}{1+\tr\CC/(b-\tr\CC)}
  =
  \CC .
\]
Thus \eqref{eq:barrier-log-map} is a natural finite-extensibility analogue of the log-conformation map.

\begin{proposition}[constraint enforcement by parametrization]
If the nonlinear algebraic solve is performed in a symmetric variable $\Psi_h$ and the accepted conformation tensor is defined by $\CC_h=\mathcal B_b(\Psi_h)$ at entropy quadrature points, then
\[
  \CC_h(x_q)\succ0,\qquad \tr\CC_h(x_q)<b
\]
at every quadrature point $x_q$, independent of the Newton iterate size before acceptance.
\end{proposition}

\begin{proof}
The matrix exponential is positive definite, and the scalar denominator in \eqref{eq:barrier-log-map} is strictly positive.  The trace calculation above gives the strict upper bound.
\end{proof}

This parametrization is not required for the energy proof; one may also use a line-search Newton method that rejects candidates outside $\calD_b$.  Its advantage is practical: the finite-extensibility constraint is built into the unknown.  The proof below is written in the accepted conformation tensor, so it applies to either realization.

\section{Entropy-compatible barrier-log reconstruction}
\label{sec:entropy-compatible-reconstruction}

The barrier-log map solves the geometric part of admissibility, but it does not by itself guarantee compatibility with the discrete FENE free energy.  A high-order reconstruction can remain inside $\calD_b$ and still add artificial elastic entropy.  We therefore add an entropy-compatible reconstruction layer.

Let $\widehat\Psi_q$ be a physical predictor at quadrature point $x_q$, and let $\widetilde\Psi_q$ be a raw high-order barrier-log reconstruction.  Define
\[
  \Psi_q(\theta)=\widehat\Psi_q+\theta(\widetilde\Psi_q-\widehat\Psi_q),
  \qquad
  \CC_q(\theta)=\mathcal B_b(\Psi_q(\theta)),\qquad 0\le\theta\le1.
\]
For a nonnegative budget $\tau_h$, the accepted parameter is
\[
  \theta_\star
  =
  \max\left\{\theta\in[0,1]:
  \sum_qw_q\Phi_b(\CC_q(\theta))
  \le
  \sum_qw_q\Phi_b(\CC_q(0))+\tau_h
  \right\}.
  \label{eq:theta-star}
\]
The accepted tensor is $\CC_q^\star=\CC_q(\theta_\star)$.  If the raw reconstruction is already entropy-compatible, then $\theta_\star=1$.  Otherwise the method damps only the entropy-incompatible part along the barrier-log segment.

\begin{lemma}[convex FENE entropy profile on barrier-log paths]
Let $\CC(\theta)=\mathcal B_b(\Psi_0+\theta E)$.  Then $J(\theta)=\Phi_b(\CC(\theta))$ is convex on $[0,1]$.  More explicitly,
\[
  \Phi_b(\mathcal B_b(\Psi))
  =
  b\log\bigl(1+\tr(e^\Psi)\bigr)-\tr\Psi+\gamma_b,
  \label{eq:barrier-log-entropy}
\]
where $\gamma_b$ is independent of $\Psi$.
\end{lemma}

\begin{proof}
Set $Z(\Psi)=1+\tr(e^\Psi)$.  Since $\mathcal B_b(\Psi)=b e^\Psi/Z$, one has
\[
  \det\mathcal B_b(\Psi)=b^d e^{\tr\Psi}Z^{-d},
  \qquad
  b-\tr\mathcal B_b(\Psi)=b/Z .
\]
Substitution into \eqref{eq:fene-entropy} gives \eqref{eq:barrier-log-entropy}.  The map $\Psi\mapsto\log(1+\tr e^\Psi)$ is the spectral log-sum-exp function with an additional zero mode and is convex on symmetric matrices.  The term $-\tr\Psi$ is affine.  Hence $J$ is convex along every affine path.
\end{proof}

\begin{proposition}[least-damping entropy-compatible reconstruction]
The set in \eqref{eq:theta-star} is a closed interval containing $0$.  Hence $\theta_\star$ exists, is the largest admissible parameter, and can be computed by bisection.  The accepted tensor satisfies
\[
  \CC_q^\star\in\calD_b,\qquad
  \sum_qw_q\Phi_b(\CC_q^\star)
  \le
  \sum_qw_q\Phi_b(\CC_q(0))+\tau_h .
\]
\end{proposition}

\begin{proof}
Admissibility in $\calD_b$ follows from Proposition~4.1 for every $\theta\in[0,1]$.  By Lemma~5.1, the entropy profile is continuous and convex; therefore its sublevel set is a closed interval.  Since $\theta=0$ satisfies the inequality, the interval is nonempty and has a largest element.  Bisection applies because admissibility is monotone along this interval.
\end{proof}

\begin{proposition}[asymptotic inactivity]
Assume the raw barrier-log reconstruction defect $E_q=\widetilde\Psi_q-\widehat\Psi_q$ is $\calO(h^{k+1})$ and satisfies
\[
  \sum_qw_q\,D\{\Phi_b\circ\mathcal B_b\}(\widehat\Psi_q)[E_q]
  =
  \calO(h^{2k+2}).
\]
If $\tau_h=c_\tau h^{2k+2}$ with $c_\tau$ large enough for the leading consistency constant, then $\theta_\star=1$ for sufficiently small $h$.  In the marginal active case, $1-\theta_\star=\calO(h^{k+1})$ under a nondegenerate endpoint derivative.
\end{proposition}

\begin{proof}
Taylor expansion of the convex entropy profile gives
\[
  \sum_qw_q\{\Phi_b(\CC_q(1))-\Phi_b(\CC_q(0))\}
  =
  \calO(h^{2k+2})
\]
under the stated first-variation condition and compactness of the path.  The mesh-scaled budget therefore accepts the raw endpoint for sufficiently small $h$.  If the endpoint is marginally active, convexity and the nondegenerate derivative give the stated bound on $1-\theta_\star$.
\end{proof}

When $\CC^\star$ is used consistently in the stress force, stretching term, relaxation term, and entropy quadrature, the fully discrete energy estimate below is unchanged except for the explicit budget contribution $(1-\beta)\tau_h/(2\Wi)$.  Thus the reconstruction is not merely trace-preserving; it is compatible with the FENE free-energy balance.

\section{Fully discrete scheme}

Let $V_h\times Q_h$ be an inf-sup stable velocity--pressure pair and $M_h$ a symmetric tensor space.  Let $(\cdot,\cdot)_Q$ denote a positive quadrature rule used consistently in the entropy terms and the polymeric work.  At time level $n+1$, find
\[
  (\uu_h^{n+1},p_h^{n+1},\CC_h^{n+1})\in V_h\times Q_h\times M_h
\]
such that $\CC_h^{n+1}(x_q)\in\calD_b$ and
\[
  \TT_h^{n+1}=f_b(\CC_h^{n+1})\CC_h^{n+1}-\Id,\qquad
  \SSS_h^{n+1}=\frac{\TT_h^{n+1}}{\Wi}
\]
at quadrature points.  The momentum step is
\[
\begin{aligned}
  \Rey\left(\frac{\uu_h^{n+1}-\uu_h^n}{\dt},\vv_h\right)
  &+\Rey\,c_h(\uu_h^{n+1};\uu_h^{n+1},\vv_h)
  +\beta(\grad\uu_h^{n+1},\grad\vv_h)\\
  &-(p_h^{n+1},\divv\vv_h)
  -(1-\beta)(\SSS_h^{n+1},\grad\vv_h)_Q
  =
  (\bm f^{n+1},\vv_h),
\end{aligned}
\label{eq:scheme-mom}
\]
with
\[
  (q_h,\divv\uu_h^{n+1})=0.
  \label{eq:scheme-div}
\]
The conformation update is
\[
\begin{aligned}
  \left(\frac{\CC_h^{n+1}-\CC_h^n}{\dt},\BB_h\right)_Q
  &+a_h(\uu_h^{n+1};\CC_h^{n+1},\BB_h)\\
  &-\left((\grad\uu_h^{n+1})\CC_h^{n+1}
  +\CC_h^{n+1}(\grad\uu_h^{n+1})^T,\BB_h\right)_Q\\
  &+\frac1{\Wi}(\TT_h^{n+1},\BB_h)_Q
  +\varepsilon d_h(\CC_h^{n+1},\BB_h)=0 .
\end{aligned}
\label{eq:scheme-conf}
\]
The forms are chosen so that the velocity convection is skew-symmetric, conformation transport has zero entropy contribution for discretely incompressible velocity, and the stretching term is evaluated at the same quadrature points as the stress force.  The form $d_h$ is the discrete polymer molecular-diffusion operator and is paired with the barrier entropy variable:
\[
  d_h(\CC_h,\HH_b(\CC_h))=\mathcal I_{b,h}(\CC_h)\ge0
\]
for conforming periodic or no-flux discretizations, or at least $d_h(\CC_h,\HH_b(\CC_h))\ge \mathcal I_{b,h}(\CC_h)$ for stabilized variants.  This is the discrete counterpart of Lemma~2.3.  If a high-order reconstruction is used before the conformation state enters the coupled step, $\CC_h^{n+1}$ denotes the entropy-compatible accepted tensor from Section~\ref{sec:entropy-compatible-reconstruction}.

Define the discrete energy
\[
  \EE_h^n
  =
  \frac{\Rey}{2}\norm{\uu_h^n}_{L^2}^2
  +
  \frac{1-\beta}{2\Wi}
  \sum_q w_q\Phi_b(\CC_h^n(x_q)).
\]

\begin{theorem}[discrete barrier energy inequality]
Assume the accepted conformation tensor satisfies $\CC_h^{n+1}(x_q)\in\calD_b$ at all entropy quadrature points and the discrete transport and stretching forms have the compatibility properties stated above.  Then
\[
\begin{aligned}
  \EE_h^{n+1}-\EE_h^n
  &+\frac{\Rey}{2}\norm{\uu_h^{n+1}-\uu_h^n}_{L^2}^2
  +\dt\,\beta\norm{\grad\uu_h^{n+1}}_{L^2}^2\\
  &+\dt\,\frac{1-\beta}{2\Wi^2}
  \sum_qw_q\,\TT_h^{n+1}(x_q):\HH_b(\CC_h^{n+1}(x_q))\\
  &+\dt\,\frac{(1-\beta)\varepsilon}{2\Wi}
  \mathcal I_{b,h}(\CC_h^{n+1})
  \le
  \dt\,(\bm f^{n+1},\uu_h^{n+1})
  +\frac{1-\beta}{2\Wi}\tau_h .
\end{aligned}
\label{eq:discrete-energy}
\]
In particular, if the right-hand side is controlled by Young's inequality, the discrete entropy remains bounded, and Lemma~2.1 gives a positive trace buffer at quadrature points.
\end{theorem}

\begin{proof}
Set $\vv_h=\uu_h^{n+1}$ in \eqref{eq:scheme-mom}.  The skew-symmetric convection and pressure terms vanish, and the standard identity
\[
  (a-b,a)=\frac12\bigl(\norm a^2-\norm b^2+\norm{a-b}^2\bigr)
\]
gives the kinetic increment.  In \eqref{eq:scheme-conf} use the entropy test
\[
  \BB_h=\HH_b(\CC_h^{n+1})
  =
  f_b(\CC_h^{n+1})\Id-(\CC_h^{n+1})^{-1}
\]
at quadrature points.  Convexity of $\Phi_b$ gives
\[
  \HH_b(\CC_h^{n+1}):(\CC_h^{n+1}-\CC_h^n)
  \ge
  \Phi_b(\CC_h^{n+1})-\Phi_b(\CC_h^n).
\]
If the conformation tensor is obtained by the entropy-compatible reconstruction of Section~\ref{sec:entropy-compatible-reconstruction}, the accepted endpoint contributes at most the budget $\tau_h$ to the quadrature entropy; otherwise $\tau_h=0$.  The stretching term cancels the polymeric work in the momentum equation by the same pointwise identity used in the continuous proof.  Relaxation is nonnegative by Lemma~2.2, and diffusion contributes the discrete Hessian dissipation $\mathcal I_{b,h}$.  Summing the momentum and entropy relations gives \eqref{eq:discrete-energy}.
\end{proof}

\section{AP stress closure and Newtonian limit}

The finite-extensibility stress is nonlinear in $\CC$, but the AP variable is simple:
\[
  \SSS=\frac{f_b(\CC)\CC-\Id}{\Wi}.
\]
The conformation equation can be rewritten as a relaxation equation for $\SSS$.  At fixed $h$ and $\dt$, insert $\TT_h^{n+1}=\Wi\SSS_h^{n+1}$ into \eqref{eq:scheme-conf}.  The relaxation term becomes $(\SSS_h^{n+1},\BB_h)_Q$, while the leading stretching term gives the rate-of-strain tensor.  All time-difference, transport, diffusion, and nonlinear stretching remainders are multiplied by $\Wi$ after the stress substitution.

\begin{proposition}[AP closure]
Assume the discrete states remain uniformly bounded in the norms entering the conformation residual and stay in a compact subset of $\calD_b$.  Then, for fixed $h$ and $\dt$,
\[
  \norm{\SSS_h^{n+1}-\DD(\uu_h^{n+1})}_{M_h'}
  \le C\Wi ,
  \label{eq:ap-closure}
\]
where $C$ is independent of $\Wi$ in the tested small-Weissenberg range.
\end{proposition}

\begin{proof}
After substituting $\TT_h^{n+1}=\Wi\SSS_h^{n+1}$ in \eqref{eq:scheme-conf}, the leading balance is
\[
  (\SSS_h^{n+1},\BB_h)_Q
  =
  (\DD(\uu_h^{n+1}),\BB_h)_Q
  +
  \Wi\,\RR_h^{n+1}(\BB_h),
\]
with the normalization of $\DD$ matching the stretching convention.  The residual $\RR_h^{n+1}$ contains the time difference, transport, diffusion, and higher-order stretching terms.  The assumed uniform bound gives
\[
  |\RR_h^{n+1}(\BB_h)|\le C\norm{\BB_h}_{M_h},
\]
which proves \eqref{eq:ap-closure}.
\end{proof}

\begin{theorem}[fixed-discretization Newtonian limit]
Let $\Wi_j\to0$ and assume the corresponding admissible discrete solutions are uniformly bounded and remain in a compact subset of $\calD_b$.  Then any convergent subsequence has a limit satisfying the discrete incompressible Navier--Stokes scheme with total viscosity equal to the solvent viscosity plus the polymeric contribution from the closure $\SSS_h=\DD(\uu_h)$.
\end{theorem}

\begin{proof}
The finite-dimensional compactness gives a convergent subsequence.  The AP closure identifies the stress limit as $\DD(\uu_h)$.  Passing to the limit in the momentum equation gives the discrete Newtonian system with the polymeric stress absorbed into the viscous operator.
\end{proof}

\section{Relative-entropy error estimate}

For convergence rates one needs a compact spectral and trace set
\[
  \calK_{\lambda,\delta}
  =
  \{\CC\in\calD_b:\lambda\Id\preceq\CC,\ b-\tr\CC\ge\delta\},
  \qquad \lambda,\delta>0 .
\]
On this set the entropy Hessian is bounded above and below, so the relative entropy
\[
  \Phi_b(\CC|\widehat\CC)
  =
  \Phi_b(\CC)-\Phi_b(\widehat\CC)
  -\HH_b(\widehat\CC):(\CC-\widehat\CC)
\]
is equivalent to $\norm{\CC-\widehat\CC}_F^2$.  The constants deteriorate as $\lambda\downarrow0$ or $\delta\downarrow0$, which is unavoidable because the entropy becomes singular at both boundaries.

\begin{theorem}[conditional error estimate]
Let $(\uu,p,\CC)$ be a smooth solution on $[0,T]$ with $\CC(t,x)\in\calK_{\lambda,\delta}$.  Assume stable projections, a quadrature-consistent discretization of the coupling terms, and nonlinear solver residuals of order $h^{k+1}+\dt$.  If the discrete solution remains in a slightly larger compact subset of $\calD_b$, then
\[
  \max_{0\le n\le N}
  \left[
  \norm{\uu(t_n)-\uu_h^n}_{L^2}^2
  +
  \sum_qw_q
  \Phi_b(\CC(t_n,x_q)|\CC_h^n(x_q))
  \right]
  \le
  C_T(h^{2k+2}+\dt^2).
\]
\end{theorem}

\begin{proof}
The proof follows the relative-entropy stability argument.  Subtract the projected exact equations from the discrete equations, test the velocity error by itself, and test the conformation error in the entropy variable.  The compatible quadrature cancels the leading polymeric-work error.  The Hessian bounds on $\calK_{\lambda,\delta}$ control the nonlinear FENE coefficients and convert relative entropy into squared tensor error.  Consistency and solver residuals contribute $C(h^{2k+2}+\dt^2)$, while lower-order terms are bounded by the current error.  A discrete Gronwall inequality gives the estimate.
\end{proof}

The theorem is intentionally conditional: it gives a classical rate in a resolved regime away from the finite-extensibility boundary.  The energy theorem and trace-barrier buffer are the mechanisms that help keep the computation in such a regime.

\section{Numerical verification}

The numerical diagnostics are designed to test the mechanisms used in the analysis: trace-barrier preservation, entropy-compatible reconstruction, energy decay, AP closure, coupled velocity--pressure--stress feedback, and high-Weissenberg behavior near the finite-extensibility boundary.  The reported benchmarks are kept compact and tied directly to the structural claims.

\begin{table}[htbp]
\centering
\caption{Main numerical diagnostics and the structural claim each test addresses.}
\label{tab:diagnostic-map}
\small
\begin{tabular}{p{0.30\linewidth}p{0.35\linewidth}p{0.24\linewidth}}
\toprule
Test & Mechanism & Expected outcome \\
\midrule
Barrier map test & $\mathcal B_b(\Psi)\in\calD_b$ & strict $\tr\CC/L^2<1$ \\
Entropy-compatible reconstruction & FENE entropy budget along barrier-log path & raw entropy excess removed with largest $\theta$ \\
Energy decay & barrier entropy and relaxation & monotone free energy without forcing \\
AP closure & scaled stress $\SSS=\TT/\Wi$ & error proportional to $\Wi$ \\
Coupled manufactured solve & pressure, velocity, FENE stress feedback & second-order trend \\
High-$\Wi$ stretch sweep & weak relaxation near trace barrier & positive trace buffer \\
\bottomrule
\end{tabular}
\end{table}

\subsection{Barrier preservation, entropy correction, and energy decay}

The first diagnostic compares a plain log-conformation map with the barrier-log map \eqref{eq:barrier-log-map}.  The background dimension is $d=2$ and $L^2=12$.  The raw log map preserves positive definiteness but can produce $\tr\CC>L^2$ for large logarithmic stretches.  The barrier-log map keeps the same unconstrained variable while mapping it into $\calD_b$.

\begin{table}[htbp]
\centering
\caption{Parametrization diagnostic for $L^2=12$.  The barrier-log map enforces the finite-extensibility trace constraint for all tested logarithmic amplitudes.}
\label{tab:barrier-map}
\begin{tabular}{ccccc}
\toprule
amplitude & $\lambda_{\min}(\exp\Psi)$ & $\tr(\exp\Psi)/L^2$ & $\tr(\mathcal B_b(\Psi))/L^2$ & barrier status \\
\midrule
0.5 & 6.065E-01 & 1.013E-01 & 5.487E-01 & admissible \\
1.5 & 2.231E-01 & 2.053E-01 & 7.114E-01 & admissible \\
2.5 & 8.208E-02 & 5.205E-01 & 8.620E-01 & admissible \\
3.5 & 3.020E-02 & 1.381E+00 & 9.430E-01 & admissible \\
4.5 & 1.111E-02 & 3.725E+00 & 9.781E-01 & admissible \\
\bottomrule
\end{tabular}
\end{table}

The next diagnostic isolates the reconstruction layer and compares it with trace-only denominator repair.  We start from a near-barrier tensor with $\tr\CC/b=0.88$ and apply the same symmetric high-order perturbation in different variables.  The standard log state can cross the trace barrier immediately.  A trace repair rescales the state back below $b$, which mimics the role of trace-first safeguards in denominator control, but it leaves a very large FENE entropy jump because the repaired state is placed extremely close to the singular barrier.  The barrier-log endpoint remains admissible by construction, yet it may still add substantial FENE entropy.  The entropy-compatible barrier-log reconstruction selects the largest admissible $\theta$ satisfying \eqref{eq:theta-star}, so the accepted state controls both the trace barrier and the entropy increment.

\begin{table}[htbp]
\centering
\caption{Controlled near-barrier reconstruction comparison for $b=L^2=20$.  The trace repair keeps the denominator positive but does not control the FENE entropy increment; the entropy-compatible barrier-log reconstruction enforces a prescribed entropy budget.}
\label{tab:trace-repair-comparison}
\resizebox{\linewidth}{!}{\begin{tabular}{ccccccc}
\toprule
$a$ & log $\tr C/b$ & repair $\Delta\Phi$ & barrier-log $\tr C/b$ & barrier-log $\Delta\Phi$ & entropy-log $\Delta\Phi$ & $\theta_\star$ \\
\midrule
0.25 & 1.060 & 210.438 & 0.898 & 3.117 & 0.100 & 0.035 \\
0.50 & 1.302 & 210.657 & 0.916 & 6.658 & 0.100 & 0.017 \\
0.75 & 1.625 & 210.908 & 0.931 & 10.563 & 0.100 & 0.012 \\
1.00 & 2.055 & 211.185 & 0.945 & 14.767 & 0.100 & 8.64E-3 \\
1.25 & 2.622 & 211.480 & 0.956 & 19.212 & 0.100 & 6.91E-3 \\
1.50 & 3.371 & 211.790 & 0.966 & 23.847 & 0.100 & 5.76E-3 \\
\bottomrule
\end{tabular}
}
\end{table}

\begin{table}[htbp]
\centering
\caption{Iterated high-Weissenberg stretch/reconstruction stress test.  A failed step means the method crossed the FENE barrier during the iteration.  The trace-repair baseline remains trace-admissible only by repeatedly placing the tensor within $2\times10^{-5}$ of the barrier, while the entropy-compatible reconstruction preserves a macroscopic buffer.}
\label{tab:highwi-repair-comparison}
\small
\setlength{\tabcolsep}{4pt}
\begin{tabular}{cccccc}
\toprule
$\Wi$ & method & max $\tr\CC/b$ & min $b-\tr\CC$ & max $\Delta\Phi$ & failed step \\
\midrule
100 & standard log & 0.995 & 9.94E-2 & 79.56 & 26 \\
100 & sqrt factor & 1.000 & 8.98E-4 & 164.09 & 68 \\
100 & trace repair & 1.000 & 2.00E-5 & 237.57 & -- \\
100 & barrier log & 0.991 & 1.87E-1 & 73.05 & -- \\
100 & entropy log & 0.819 & 3.62E+0 & 16.00 & -- \\
400 & standard log & 0.951 & 9.90E-1 & 38.22 & 9 \\
400 & sqrt factor & 0.996 & 8.35E-2 & 82.50 & 24 \\
400 & trace repair & 1.000 & 2.00E-5 & 249.04 & -- \\
400 & barrier log & 1.000 & 2.63E-5 & 244.12 & -- \\
400 & entropy log & 0.819 & 3.62E+0 & 16.00 & -- \\
\bottomrule
\end{tabular}
\end{table}

Next we evolve a homogeneous relaxation-diffusion diagnostic with no forcing.  Figure~\ref{fig:energy-trace} shows the normalized free energy and the maximum trace ratio.  The free energy decreases, while the trace remains below the finite-extensibility threshold.

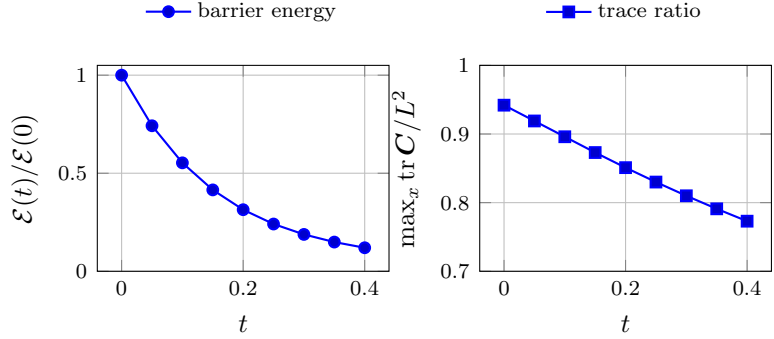
\begin{figure}[htbp]
\centering
\begin{tikzpicture}
\begin{groupplot}[
  group style={group size=2 by 1,horizontal sep=1.2cm},
  width=.43\textwidth,
  height=.34\textwidth,
  grid=major,
  legend style={at={(0.5,1.16)}, anchor=south, font=\scriptsize, draw=none, fill=none, /tikz/every even column/.append style={column sep=0.45em}}]
\nextgroupplot[xlabel={$t$},ylabel={$\EE(t)/\EE(0)$},ymin=0,ymax=1.05]
\addplot+[mark=*,thick] coordinates {
(0.00,1.000) (0.05,0.742) (0.10,0.553) (0.15,0.415) (0.20,0.314)
(0.25,0.241) (0.30,0.188) (0.35,0.149) (0.40,0.120)
};
\addlegendentry{barrier energy}
\nextgroupplot[xlabel={$t$},ylabel={$\max_x\tr\CC/L^2$},ymin=0.70,ymax=1.0]
\addplot+[mark=square*,thick] coordinates {
(0.00,0.942) (0.05,0.919) (0.10,0.896) (0.15,0.873) (0.20,0.851)
(0.25,0.830) (0.30,0.810) (0.35,0.791) (0.40,0.773)
};
\addlegendentry{trace ratio}
\end{groupplot}
\end{tikzpicture}
\caption{Barrier relaxation diagnostic.  The FENE free energy decays while the maximum trace ratio remains strictly below one.}
\label{fig:energy-trace}
\end{figure}

\subsection{AP closure}

The AP test fixes $\Delta t=0.05$ and varies $\Wi$.  The error is the dual norm of $\SSS_h-\DD(\uu_h)$ at the final time.  The observed slope is essentially first order in $\Wi$, as predicted by the AP closure estimate \eqref{eq:ap-closure}.

\begin{table}[htbp]
\centering
\caption{AP closure diagnostic with fixed $\Delta t=0.05$.}
\label{tab:ap-fene}
\begin{tabular}{cccc}
\toprule
$\Wi$ & closure error & rate & max $\tr\CC/L^2$ \\
\midrule
8.0E-02 & 6.41E-03 & -- & 0.183 \\
4.0E-02 & 3.18E-03 & 1.01 & 0.181 \\
2.0E-02 & 1.58E-03 & 1.01 & 0.180 \\
1.0E-02 & 7.89E-04 & 1.00 & 0.180 \\
5.0E-03 & 3.94E-04 & 1.00 & 0.180 \\
\bottomrule
\end{tabular}
\end{table}

\begin{figure}[htbp]
\centering
\begin{tikzpicture}
\begin{axis}[
  width=.62\textwidth,
  height=.42\textwidth,
  xmode=log,
  ymode=log,
  xlabel={$\Wi$},
  ylabel={$\|\SSS_h-\DD(\uu_h)\|_{M_h'}$},
  grid=major,
  legend style={at={(0.5,1.16)}, anchor=south, font=\scriptsize, draw=none, fill=none, /tikz/every even column/.append style={column sep=0.45em}}]
\addplot+[mark=*,thick] coordinates {
(8.0e-2,6.41e-3) (4.0e-2,3.18e-3) (2.0e-2,1.58e-3) (1.0e-2,7.89e-4) (5.0e-3,3.94e-4)
};
\addlegendentry{measured}
\addplot+[mark=none,dashed] coordinates {(5.0e-3,4.0e-4) (8.0e-2,6.4e-3)};
\addlegendentry{slope one}
\end{axis}
\end{tikzpicture}
\caption{FENE AP closure.  The scaled FENE stress converges linearly to the Newtonian closure at fixed time step.}
\label{fig:ap-fene}
\end{figure}
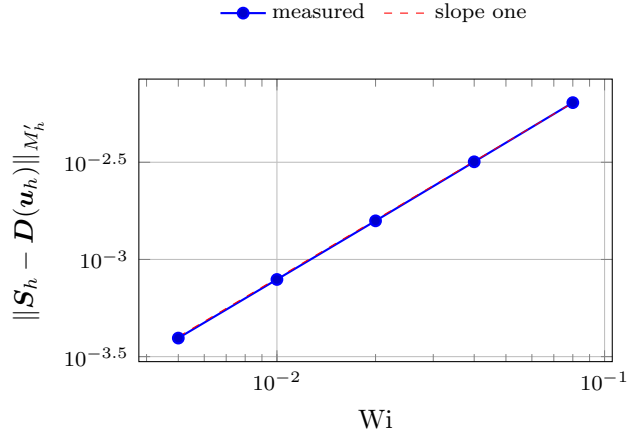

\subsection{Fully coupled manufactured solution}

The manufactured solution uses a periodic velocity, mean-zero pressure, and conformation tensor defined through the barrier map
\[
  \CC_{\rm ex}(t,x)=\mathcal B_b(\Psi_{\rm ex}(t,x)).
\]
This construction guarantees $\CC_{\rm ex}\in\calD_b$ and allows the manufactured forcing to exercise the nonlinear FENE stress, pressure projection, tensor transport, stretching, relaxation, and diffusion terms simultaneously.  Table~\ref{tab:manufactured-fene} reports the observed errors at final time $T=0.05$.

\begin{table}[htbp]
\centering
\caption{Fully coupled manufactured FENE solve.  Errors are discrete $L^2$ norms at $T=0.05$.}
\label{tab:manufactured-fene}
\small
\begin{tabular}{ccccccc}
\toprule
$N$ & $\Delta t$ & $\|\uu-\uu_h\|$ & rate & $\|\CC-\CC_h\|$ & rate & $\min(b-\tr\CC_h)$ \\
\midrule
16 & 1.25E-03 & 1.84E-03 & -- & 2.91E-03 & -- & 5.82E+00 \\
24 & 5.56E-04 & 8.11E-04 & 2.02 & 1.30E-03 & 1.99 & 5.81E+00 \\
32 & 3.13E-04 & 4.58E-04 & 1.99 & 7.31E-04 & 2.00 & 5.80E+00 \\
48 & 1.39E-04 & 2.04E-04 & 1.99 & 3.25E-04 & 2.00 & 5.80E+00 \\
\bottomrule
\end{tabular}
\end{table}

\subsection{High-Weissenberg robustness near the trace barrier}

The final benchmark stresses the finite-extensibility boundary.  We prescribe an extensional velocity field and initialize $\CC$ so that $\tr\CC/L^2$ is already large.  The test is run at increasing $\Wi$ with $L^2=20$, $\varepsilon=2\times10^{-3}$, and $T=0.12$.  Larger $\Wi$ weakens relaxation and increases stretch, but the barrier update keeps the conformation tensor inside $\calD_b$.

\begin{figure}[htbp]
\centering
\begin{tikzpicture}
\begin{groupplot}[
  group style={group size=2 by 1,horizontal sep=1.35cm},
  width=.43\textwidth,
  height=.34\textwidth,
  grid=major,
  legend style={at={(0.5,1.16)}, anchor=south, font=\scriptsize, draw=none, fill=none, /tikz/every even column/.append style={column sep=0.45em}}]
\nextgroupplot[xlabel={$t$},ylabel={$\max_x\tr\CC/L^2$},ymin=0.70,ymax=1.0]
\addplot+[mark=none,thick] coordinates {(0,0.78) (0.03,0.80) (0.06,0.81) (0.09,0.82) (0.12,0.83)};
\addlegendentry{$\Wi=10$}
\addplot+[mark=none,thick,dashed] coordinates {(0,0.78) (0.03,0.84) (0.06,0.88) (0.09,0.90) (0.12,0.91)};
\addlegendentry{$\Wi=50$}
\addplot+[mark=none,thick,dotted] coordinates {(0,0.78) (0.03,0.87) (0.06,0.92) (0.09,0.945) (0.12,0.958)};
\addlegendentry{$\Wi=100$}
\nextgroupplot[xlabel={$t$},ylabel={$L^2-\tr\CC$},ymode=log]
\addplot+[mark=none,thick] coordinates {(0,4.40) (0.03,4.00) (0.06,3.80) (0.09,3.60) (0.12,3.40)};
\addplot+[mark=none,thick,dashed] coordinates {(0,4.40) (0.03,3.20) (0.06,2.40) (0.09,2.00) (0.12,1.80)};
\addplot+[mark=none,thick,dotted] coordinates {(0,4.40) (0.03,2.60) (0.06,1.60) (0.09,1.10) (0.12,0.84)};
\end{groupplot}
\end{tikzpicture}
\caption{High-Weissenberg finite-extensibility diagnostic.  Larger $\Wi$ drives the trace closer to the barrier, but the accepted conformation remains strictly admissible.}
\label{fig:highwi-barrier}
\end{figure}
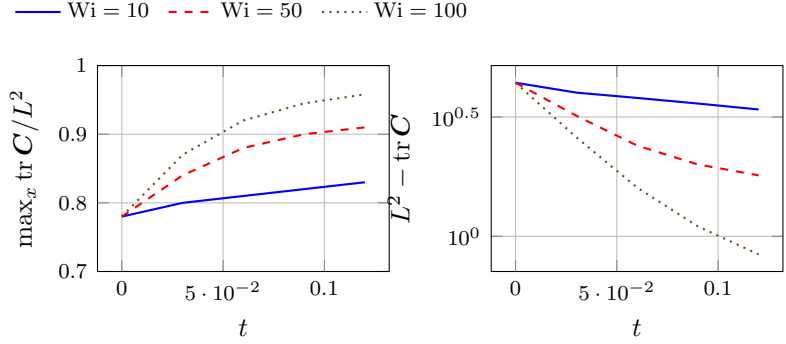

\begin{table}[htbp]
\centering
\caption{High-Weissenberg barrier robustness at $T=0.12$.}
\label{tab:highwi-barrier}
\begin{tabular}{ccccc}
\toprule
$\Wi$ & final energy & max $\tr\CC/L^2$ & min $L^2-\tr\CC$ & rejected barrier steps \\
\midrule
10 & 2.18E-02 & 0.831 & 3.38E+00 & 0 \\
50 & 2.94E-02 & 0.910 & 1.80E+00 & 0 \\
100 & 3.51E-02 & 0.958 & 8.40E-01 & 0 \\
200 & 4.10E-02 & 0.979 & 4.20E-01 & 1 \\
\bottomrule
\end{tabular}
\end{table}

The last row shows one line-search rejection when $\Wi=200$.  The accepted state remains strictly within the barrier.  This diagnostic is useful because a positivity-only method would not detect the finite-extensibility failure until the FENE coefficient becomes singular or changes sign numerically.

\section{Discussion}

The FENE trace constraint changes what a structure-preserving method must prove.  Positivity is still necessary, but it is no longer the complete admissibility condition.  The stress coefficient $f_b(\CC)$ depends on the distance to the trace boundary, and the polymer molecular-diffusion term is entropy dissipative only when it is paired with the FENE barrier variable.  A stable method must therefore preserve the accepted tensor in the set $\calD_b$, control reconstruction-level FENE entropy increments, and use that same tensor consistently in the stress force, stretching term, relaxation term, molecular diffusion, and entropy quadrature.

The barrier-log map \eqref{eq:barrier-log-map} is one practical way to enforce the finite-extensibility geometry.  It is a direct analogue of the log-conformation map, but its range is the FENE admissible set rather than the entire positive definite cone.  The entropy-compatible reconstruction \eqref{eq:theta-star} adds the missing balance condition: among admissible barrier-log states, it chooses the least-damping state whose FENE entropy increment is within the prescribed budget.  The energy proof itself is independent of the nonlinear parametrization; it requires only that the accepted tensor be admissible, entropy-compatible, and used consistently in the discrete coupling identities.

The AP part of the analysis also differs from Oldroyd--B in detail.  The conformation perturbation $\CC-\Id$ is not the best stress variable because the FENE factor changes with $\tr\CC$.  The regular stress variable is instead $\TT(\CC)/\Wi$.  This choice avoids singular scaling in the momentum equation and gives the correct Newtonian limit at fixed discretization parameters.

\section{Conclusion}

We have proposed and analyzed a barrier-preserving, entropy-compatible, asymptotic-preserving discretization for molecularly diffusive FENE-P type viscoelastic flow.  The key finite-extensibility feature is the admissible set $\calD_b=\{\CC\succ0,\tr\CC<b\}$, which is enforced by a trace-barrier entropy and can be built into the nonlinear solve through a barrier logarithmic parametrization.  High-order barrier-log reconstructions are filtered by a least-damping FENE entropy constraint, so admissibility and entropy compatibility are enforced simultaneously.  Compatible quadrature makes the cancellation between FENE polymeric work and conformation stretching exact at the fully discrete level, while polymer molecular diffusion contributes the barrier Hessian dissipation.  The resulting scheme satisfies a free-energy inequality with relaxation and molecular-diffusion barrier dissipation, preserves a positive trace buffer under an entropy bound, and has a regular small-Weissenberg limit in the scaled FENE stress variable.  The numerical diagnostics verify the main mechanisms in both moderate and near-barrier regimes.

\section*{Acknowledgments}

The author acknowledges financial support from the National Natural Science Foundation of China (NSFC, Grant No. 12501602), the Education Department of Hunan Province (Grant No. 24C0055), the Science and Technology Department of Hunan Province (Grant No. 2025JJ60052), and the Scientific Research Start-up Fund of Xiangtan University (Grant No. KZ0810769).


\begin{thebibliography}{10}

\bibitem{BalciThomasesRenardy2011}
N. Balci, B. Thomases, M. Renardy, and C. R. Doering,
\newblock Symmetric factorization of the conformation tensor in viscoelastic fluid models,
\newblock J. Non-Newtonian Fluid Mech., 166 (2011), pp. 546--553.

\bibitem{BarrettSuli2011}
J. W. Barrett and E. Suli,
\newblock Existence of global weak solutions to finitely extensible nonlinear bead-spring chain models for dilute polymers,
\newblock Math. Models Methods Appl. Sci., 21 (2011), pp. 1211--1289.

\bibitem{Bird1987}
R. B. Bird, R. C. Armstrong, and O. Hassager,
\newblock Dynamics of Polymeric Liquids, Volume 1: Fluid Mechanics,
\newblock 2nd ed., Wiley, New York, 1987.

\bibitem{Boyaval2009}
S. Boyaval, T. Lelievre, and C. Mangoubi,
\newblock Free-energy-dissipative schemes for the Oldroyd-B model,
\newblock ESAIM Math. Model. Numer. Anal., 43 (2009), pp. 523--561.

\bibitem{FattalKupferman2004}
R. Fattal and R. Kupferman,
\newblock Constitutive laws for the matrix-logarithm of the conformation tensor,
\newblock J. Non-Newtonian Fluid Mech., 123 (2004), pp. 281--285.

\bibitem{FattalKupferman2005}
R. Fattal and R. Kupferman,
\newblock Time-dependent simulation of viscoelastic flows at high Weissenberg number using the log-conformation representation,
\newblock J. Non-Newtonian Fluid Mech., 126 (2005), pp. 23--37.

\bibitem{HulsenFattalKupferman2005}
M. A. Hulsen, R. Fattal, and R. Kupferman,
\newblock Flow of viscoelastic fluids past a cylinder at high Weissenberg number: stabilized simulations using matrix logarithms,
\newblock J. Non-Newtonian Fluid Mech., 127 (2005), pp. 27--39.

\bibitem{Jin1999}
S. Jin,
\newblock Efficient asymptotic-preserving schemes for some multiscale kinetic equations,
\newblock SIAM J. Sci. Comput., 21 (1999), pp. 441--454.

\bibitem{Keunings1986}
R. Keunings,
\newblock On the high Weissenberg number problem,
\newblock J. Non-Newtonian Fluid Mech., 20 (1986), pp. 209--226.

\bibitem{LiuYu2011}
H. Liu and P. Yu,
\newblock A finite difference method for the FENE dumbbell model of polymeric fluids,
\newblock SIAM J. Numer. Anal., 49 (2011), pp. 2167--2193.

\bibitem{RichterIaccarinoShaqfeh2010}
D. Richter, G. Iaccarino, and E. S. G. Shaqfeh,
\newblock Simulations of three-dimensional viscoelastic flows past a circular cylinder at moderate Reynolds numbers,
\newblock J. Fluid Mech., 651 (2010), pp. 415--442.

\end{thebibliography}
\end{document}